\documentclass[11pt, oneside]{amsart}
\usepackage{wrapfig}
\usepackage{tikz}
\usepackage{comment}
\usepackage{mathrsfs}
\usepackage{tabularx, hyperref}
\usepackage{amssymb} \usepackage{amsfonts}
 \usepackage{amsmath}
\usepackage{amsthm} \usepackage{epsfig, subfig}
\usepackage{ amscd, amsxtra, latexsym}
\usepackage{epsfig,  graphicx, psfrag}
\usepackage[all]{xy}
\usepackage{caption}
\usepackage{enumerate}
\usepackage{color}
\usepackage{amssymb, amsmath,amsthm, mathtools}

\DeclareMathOperator\supp{Supp}

\DeclareMathOperator{\aut}{Aut}
\DeclareMathOperator{\mult}{mult}

\newtheorem{lemma}{Lemma}[section]
\newtheorem{thm}[lemma]{Theorem}
\newtheorem{prop}[lemma]{Proposition}

\newtheorem*{prop*}{Proposition}
\newtheorem{prop_intro}{Proposition}
\newtheorem{thm_intro}[prop_intro]{Theorem}

\theoremstyle{definition}

\newtheorem{defn}[lemma]{Definition}

\newtheorem{rem}[lemma]{Remark}

\theoremstyle{definition}

\definecolor{darkgreen}{cmyk}{1,0,1,.2}

\newcommand{\alt}{\ensuremath{\mathrm{alt}}}

\newcommand{\R} {\ensuremath {\mathbb{R}}}

\newcommand{\G} {\ensuremath {\Gamma}}

\newcommand{\calS} {\ensuremath {\mathcal{S}}}
\newcommand{\calM} {\ensuremath {\mathcal{M}}}

\newcommand{\calA} {\ensuremath {\mathcal{A}}}

\newcommand{\calU} {\ensuremath {\mathcal{U}}}

\renewcommand{\phi}{\varphi}

\newcommand{\calK} {\ensuremath {\mathcal{K}}}

\newcommand{\cal}{\mathcal}

\begin{document}

\title[]{Amenable covers and $\ell^1$-invisibility}

\author[]{R. Frigerio}
\address{Dipartimento di Matematica, Universit\`a di Pisa, Largo B. Pontecorvo 5, 56127 Pisa, Italy}
\email{roberto.frigerio@unipi.it}
\thanks{}


\begin{abstract}
Let $X$ be a topological space admitting an amenable cover of multiplicity $k\in\mathbb{N}$. We show that, for every $n\geq k$ and every $\alpha\in H_n(X;\R)$, the image
of $\alpha$ in the $\ell^1$-homology module $H_n^{\ell^1}(X;\R)$ vanishes. This strenghtens previous results by Gromov and Ivanov, who proved, under the same assumptions,
that the $\ell^1$-seminorm of $\alpha$ vanishes.
\end{abstract}
\maketitle

Gromov's Vanishing Theorem~\cite{Grom82} asserts that, if a topological space $X$ admits an amenable cover of multiplicity $k$, then for every $n\geq k$
the comparison map $H^n_b(X;\R)\to H^n(X;\R)$ between the singular bounded cohomology  and the ordinary singular cohomology of $X$  is null
(see also~\cite[Corollary 6.3]{Ivanov}, \cite[Theorem 9.2]{ivanov3} and~\cite[Theorem 6]{FriMo}).
Via a by now standard duality argument, this implies the following:

\begin{thm_intro}\label{classico}
Let $X$ be a topological space admitting an amenable open cover of multiplicity $k$ and let $n\geq k$. Then
 the $\ell^1$-seminorm of every element 
$\alpha\in H_n(X;\R)$ vanishes.
\end{thm_intro}

We say that a class $\alpha\in H_n(X;\R)$ is \emph{$\ell^1$-invisible} if $\iota_n(\alpha)=0$, where
$$
\iota_*\colon H_*(X;\R)\to H_*^{\ell^1}(X;\R)
$$
is the map induced by the inclusion of singular chains into $\ell^1$-chains (see Section~\ref{preliminary:sec}). It is 
easy to prove that 
$\ell^1$-invisible classes have vanishing $\ell^1$-seminorm (see 
Remark~\ref{contro:rem} for a proof of this fact and a brief discussion of the converse implication).

In this paper we strengthen Theorem~\ref{classico} by proving the following:

\begin{thm_intro}\label{main:thm}
Let $X$ be a topological space admitting an amenable open cover of multiplicity $k$. Then, for every $n\geq k$ and every $\alpha\in H_n(X;\R)$, the class
$\alpha$ is $\ell^1$-invisible.
\end{thm_intro}

Our argument  relies upon diffusion of chains and the theory of multicomplexes, as developed in~\cite{FriMo}.
Bounded cohomology plays a fundamental role in allowing us to work with the \emph{aspherical} multicomplex associated to $X$, thus forgetting all the higher homotopy 
groups of $X$. However, the final (and fundamental) step in the proof of Theorem~\ref{main:thm} is purely homological (see Remark~\ref{aspherical:rem}). 
In fact, as observed in~\cite[p.~258]{Loeh}, bounded cohomology can detect only whether the seminorm of a given class in 
 $\ell^1$-homology is zero, but not whether the class itself is zero. Therefore,  Theorem~\ref{main:thm} cannot be deduced from vanishing theorems for bounded cohomology. 

Our terminology is inspired by the paper~\cite{Loeh}, where a closed oriented $n$-dimensional manifold $M$ is defined to be \emph{$\ell^1$-invisible} if the real fundamental class
$[M]\in H_n(M,\R)$ is $\ell^1$-invisible. The $\ell^1$-invisibility of manifolds comes into play e.g.~when studying the simplicial volume of open manifolds: it is proved in~\cite{Loeh} that, if $N$ is a closed oriented $(n+1)$-manifold with boundary, then the simplicial volume of $N\setminus\partial N$ is finite if and only
if every component of $\partial N$ is $\ell^1$-invisible.

\section{Preliminaries}\label{preliminary:sec}

\subsection*{Amenable covers}

Let $X$ be a topological space and let $i\colon U\hookrightarrow X$ be the inclusion of a subset $U$ of $X$. Then
$U$ is \textit{amenable} (in $X$) if for every $x_0\in U$ the image of 
$i_{*} \colon \pi_{1}(U,x_0) \rightarrow \pi_{1}(X,x_0)$ is an amenable subgroup of $\pi_{1}(X,x_0)$ (the set $U$ is not assumed to be path connected).

Let $\calU=\{U_i\}_{i\in I}$ be a cover of $X$, i.e.~suppose that $U_i\subseteq X$ for every $i\in I$ and that $X=\bigcup_{i\in I} U_i$. We say that the cover
is \emph{open} if each $U_i$ is open in $X$, and \emph{amenable} if each $U_i$ is amenable in $X$. The multiplicity of $\calU$ is defined by
\begin{align*}
\mult(\calU)&=\sup \left\{\#J\, \Big|\, J\subseteq I,\, \bigcap_{i\in J} U_{i}\neq\emptyset\right\}\ , 
\end{align*}
where $\# J$ denotes the cardinality of $J$.

\subsection*{$\ell^1$-homology and bounded cohomology}
We denote by $C_* (X)$ (resp.~by $C^*(X)$) the 
complex of singular chains (resp.~cochains) on $X$ with real coefficients,  and
by $S_i (X)$ the set of singular $i$--simplices
with values in $X$.

For every $n\geq 0$ we endow the space $C_n(X)$ with the $\ell^1$-norm
$$
\Bigg\| \sum_{\sigma \in S_n(X)} a_\sigma \sigma\Bigg\|_1 =\sum_{\sigma \in S_n(X)}  |a_\sigma|
$$
(here $a_\sigma=0$ for all but a finite number of $\sigma\in S_n(X)$). 
This norm restricts to a norm on the subspace of $n$-cycles, which descends in turn to a quotient seminorm $\|\cdot \|_1$ on the singular homology
with real coefficients $H_n(X)$. 

Let $C_n^{\ell^1}(X)$ denote the metric completion of $(C_n(X),\|\cdot\|_1)$, i.e.~the space of
(possibly infinite) linear combinations 
$$
\sum_{\sigma \in S_n(X)} a_\sigma \sigma\, ,\qquad  \sum_{\sigma \in S_n(X)}  |a_\sigma|<+\infty\ .
$$
 The boundary operator $\partial_n \colon C_n(X)\to C_{n-1}(X)$
is bounded with respect to  the $\ell^1$-norm, hence it extends to a boundary operator  $\partial_n \colon C_n^{\ell^1}(X)\to C_{n-1}^{\ell^1}(X)$.
It is easily checked that $\partial_{n-1}\circ \partial_n\colon C_n^{\ell^1}(X)\to C_{n-2}^{\ell^1}(X)$ is the zero map for every $n\in\mathbb{N}$,
hence one may define the homology of the complex $C_*^{\ell^1}(X)$, which is called \emph{$\ell^1$-homology} of $X$ and is denoted by
$H_*^{\ell^1}(X)$. The norm on $C_n^{\ell^1}(X)$ induces a quotient seminorm on $H_n^{\ell^1}(X)$ for every $n\in\mathbb{N}$.
The inclusion of complexes $C_*(X)\hookrightarrow C_*^{\ell^1}(X)$ induces a  map
$$
\iota_*\colon H_*(X)\to H_*^{\ell^1}(X)\ .
$$

Since $C_n(X)$ is dense in $C_n^{\ell^1}(X)$, any
continuous functional on $C_n(X)$ uniquely extends to a continuous functional on $C_n^{\ell^1}(X)$. We may thus denote by $C^n_b(X)$ the 
\emph{topological} dual space both of $C_n(X)$ and of $C_n^{\ell^1}(X)$. Of course, being functionals on $C_n(X)$, elements of $C^n_b(X)$ are
in particular singular cochains, i.e.~the space $C^n_b(X)$ is a subspace of $C^n(X)$. More precisely, if 
 we define the $\ell^\infty$-norm of an element $\varphi\in C^n (X)$ by setting 
$$
\|\varphi\|_\infty  = \sup \left\{|\varphi (s)|\, ,\  s\in S_n (X)\right\}\in [0,\infty]\ ,
$$
then we have
$$
C^n_b(X)=\{\varphi\in C^n(X)\, , \  \|\varphi\|_\infty<+\infty\}\ .
$$
Moreover, the $\ell^\infty$-norm $\|\varphi\|_\infty$ of an element $\varphi\in C^n_b(X)$ coincides with its dual norm as a functional on $C_n(X)$ (or on $C_n^{\ell^1}(X)$). 
 
 Since
the differential takes bounded cochains to bounded cochains, $C^*_b (X)$
is a subcomplex of $C^*(X)$.
We denote by $H^*(X)$ (resp.~$H_b^*(X)$) 
the cohomology of the complex $C^*(X)$ (resp.~$C_b^*(X)$).
Of course, $H^*(X)$ is the usual real singular cohomology module of $X$, while $H_b^*(X)$
is the \emph{real bounded cohomology module} of $X$.
Just as for ($\ell^1$-)homology, the norm on $C^i_b (X)$ descends to a quotient seminorm on $H_b^*(X)$.

The inclusion of bounded cochains into possibly unbounded cochains
induces the \emph{comparison map}
$$
c^* \colon  H_b^*(X)\to H^*(X)\ .
$$ 

Any continuous map $f\colon X\to Y$ induces norm non-increasing maps 
\begin{align*}
H_n(f)&\colon H_n(X)\to H_n(Y)\ ,\\
H_n^{\ell^1}(f)&\colon H_n^{\ell^1}(X)\to H_n^{\ell^1}(Y)\ ,\\
H^n_b(f)&\colon H^n_b(Y)\to H^n_b(X)\ .
\end{align*}
If $f$ induces an isomorphism on fundamental groups, then both
$H_n^{\ell^1}(f)$ and $H^n_b(f)$ are isometric isomorphisms for every $n\in\mathbb{N}$ 
(see~\cite[p.~40]{Grom82}, \cite[Theorem 8.4]{ivanov3}, \cite[Theorem 3]{FriMo} for the map $H^n_b(f)$,
and~\cite{Loeh} for the map $H_n^{\ell^1}(f)$).

\begin{rem}\label{contro:rem}
Let $X$ be a topological space, and take $\alpha\in H_n(X)$. It is not difficult to show that, if $\iota_n(\alpha)=0$, then $\|\alpha\|_1=0$.
In fact, if $\alpha$ is $\ell^1$-invisible and $z\in C_n(X)$ is a representative of $\alpha$, then there exists $b\in C_{n+1}^{\ell^1}(X)$ such that
$z=\partial_{n+1} b$. For any given $\varepsilon>0$, we may now write $b=b_0+b_1$, where 
$b_0\in C_n(X)$ is a finite chain, and $\|b_1\|_1<\varepsilon/(n+2)$. Then $\alpha=[z-\partial_{n+1} b_0]$, and
$$
\|z-\partial_{n+1} b_0\|_1=\|\partial_{n+1} b-\partial_{n+1} b_0\|_1=\|\partial_{n+1} b_1\|_1\leq (n+2)\|b_1\|_1\leq \varepsilon\ .
$$
This shows that $\|\alpha\|_1=0$. (More in general, as observed in~\cite{Loeh}, the fact that $C_n(X)$ is dense in $C_n^{\ell^1}(X)$ for every $n\in\mathbb{N}$
implies that 
$\iota_n\colon H_n(X)\to H_n^{\ell^1}(X)$ is a seminorm-preserving map for every $n\in\mathbb{N}$.)

As far as the author
knows, the question whether
the vanishing of  $\|\alpha\|_1$  implies that $\iota_n(\alpha)=0$ is still open (see e.g.~\cite[p.~258]{Loeh}, where
it is asked whether every oriented, closed, connected manifold with vanishing simplicial volume is already $\ell^1$-invisible).

On the contrary, if $\alpha$ belongs to $H_*^{\ell^1}(X)$, then the vanishing of $\|\alpha\|_1$ does not imply in general the vanishing of $\alpha$. 
In fact, by~\cite[Theorem 2.3]{Matsu-Mor} the existence of non-trivial classes with vanishing $\ell^1$-seminorm in $H^{\ell^1}_n(X)$ is equivalent
to the existence of non-trivial classes with vanishing $\ell^\infty$-seminorm in $H^{n+1}_b(X)$. 
Thanks to~\cite[Theorems 1 and 2]{Soma2} (and the fact that the bounded cohomology
of a space is isometrically isomorphic to the bounded cohomology of its fundamental group -- see~\cite{Grom82}, \cite[Theorem 8.3]{ivanov3}), 
this implies that for every $n\in \mathbb{N}\setminus \{0,1,3\}$ there exist a topological space $X$ and a non-trivial class $\alpha\in H_n^{\ell^1}(X)$
such that $\|\alpha\|_1=0$.

In fact, while the seminorm on $H^2_b(X)$ is a norm for every topological space $X$~\cite{Matsu-Mor, Ivanov2}, the space $H^3_b(X)$ contains non-trivial classes with vanishing
seminorm if $X$ is a closed hyperbolic surface~\cite{Soma1}, if $\pi_1(X)$ is free non-abelian~\cite{Soma2}, or, more in general, if $\pi_1(X)$ is 
acylindrically hyperbolic~\cite{FFPS}. Starting from these examples, higher dimensional non-trivial classes with vanishing seminorm can then be obtained
by taking suitable cup products~\cite{Soma2}. 
\end{rem}

\subsection*{Multicomplexes}
Multicomplexes are simplicial structures introduced by Gromov in~\cite{Grom82}. 
A multicomplex is  a regular unordered $\Delta$-complex (see e.g.~\cite[p.~533--535]{hatcher}), i.e.~a $\Delta$-complex  in which the vertices of each simplex are unordered and distinct (but distinct simplices
may share the same set of vertices). In other words,
(the geometric realization of) a multicomplex
is ``a set $K$ divided
into the union of closed affine simplices $\Delta_i$, $i\in I$,
such that the intersection of any
two simplices $\Delta_i\cap \Delta_j$ is a (simplicial) subcomplex in $\Delta_i$ as well as in $\Delta_j$''~\cite{Grom82}.
We refer the reader to~\cite[Chapter 1]{FriMo} for a thorough discussion of the notion of multicomplex. 

The geometric realization $|K|$ of a  multicomplex $K$  is constructed in the very same way as the geometric realization of simplicial sets or of simplicial complexes:
one takes one geometric simplex for every combinatorial simplex of the structure, and glue these simplices to each other  according to the combinatorics of $K$
(see e.g.~\cite[Section 1.2]{FriMo}). As a topological space, $|K|$
is endowed with the weak topology associated to its decomposition into simplices. 

For every $n\in\mathbb{N}$ we denote by $K^n$ the $n$-skeleton of $K$, i.e.~the submulticomplex of $K$ obtained by taking all the $k$-dimensional simplices of $K$ for $k\leq n$.
We denote by $\aut(K)$ the group of simplicial automorphisms of $K$. 

If $f,g\colon K\to K'$ are simplicial maps between multicomplexes, we say that $f$ is simplicially homotopic to $g$
if there exists a simplicial map $F\colon K\times [0,1]\to K'$ such that $F\circ i_0=f$ and $F\circ i_1=g$, where $i_j\colon K\to K\times \{j\}\subseteq K\times [0,1]$ is the obvious inclusion
(we refer the reader to~\cite[Definitions 3.3.2 and 3.3.3]{FriMo} for the definition of the multicomplex $K\times [0,1]$ and more details on simplicial homotopies).

Let $K$ be a multicomplex. An $n$-dimensional \emph{algebraic simplex} of $K$ is a pair
$$
\sigma=(\Delta,(v_0,\ldots,v_n))\ ,
$$
where $\Delta$ is a $k$-simplex of $K$, and $\{v_0,\ldots,v_n\}$ is equal to the set of vertices of $\Delta$.
Recall that, if $\Delta$ is a $k$-simplex of a multicomplex, then $\Delta$ has exactly $k+1$ distinct vertices, hence $k\leq n$. However,
we do not require the element of the ordered $(n+1)$-tuple $(v_0,\ldots,v_n)$ to be pairwise distinct.
For $i=0,\ldots,n$, the $i$-th face of $\sigma$ is the algebraic $(n-1)$-simplex defined by
$$
\partial_n^i \sigma= (\Delta', (v_0,\ldots,\widehat{v}_i,\ldots,v_n))\ ,
$$
where $\Delta'=\Delta$ if $\{v_0,\ldots,v_n\}=\{v_0,\ldots, \widehat{v}_i,\ldots, v_n\}$, and $\Delta'$ is the unique face of $\Delta$ whose set of vertices is given by $\{v_0,\ldots,\widehat{v}_i,\ldots,v_n\}$ otherwise.

The space $C_n(K)$ of real simplicial $n$-chains on $K$
is the real vector space having the set of algebraic $n$-simplices as a basis. The map 
$$
\partial_n \colon C_n(K)\to C_{n-1}(K)\, ,\qquad \partial_n=\sum_{i=0}^n (-1)^i \partial_n^i
$$
endows $C_*(K)$ with the structure of a chain complex, whose homology $H_*(K)$ is the simplicial homology of $K$.
We put on $C_*(K)$ the obvious $\ell^1$-norm, which defines as usual an $\ell^1$-seminorm on $H_*(K)$. 
We also denote by $C_*^{\ell^1}(K)$ the metric completion of $C_*(K)$, and by $H_*^{\ell^1}(K)$ the corresponding
$\ell^1$-homology module, endowed with the quotient seminorm.

We denote by $C^*(K)$ (resp.~$C_b^*(K)$) the complex of real simplicial cochains (resp.~of bounded real simplicial cochains)
on $K$, and by $H^*(K)$ (resp.~$H_b^*(K)$) the corresponding cohomology module. 
Just as in the singular case, $C^*_b(K)$  is endowed with the $\ell^\infty$-norm dual 
to the $\ell^1$-norm on $C_*(K)$, which descends to a quotient $\ell^\infty$-seminorm on $H^*_b(K)$.

We have natural chain inclusions
$$
\phi_*\colon C_*(K)\to C_*(|K|)\, ,\qquad \phi_*^{\ell^1}\colon C_*^{\ell^1}(K)\to C_*^{\ell^1}(|K|)\ ,
$$
which induce maps
$$
H_*(\phi_*)\colon H_*(K)\to H_*(|K|)\, ,\qquad H_*(\phi_*^{\ell^1})\colon H_*^{\ell^1}(K)\to H_*^{\ell^1}(|K|)\ .
$$

\begin{prop}[{\cite[Theorem 1.4.4]{FriMo}}]\label{simpl:iso}
The map $H_*(\phi_*)$
is an isomorphism in every degree.
\end{prop}

\subsection*{The singular multicomplex}
To our purposes, the most important example of multicomplex is the \emph{singular multicomplex} associated to a topological space. 
The singular multicomplex plays in the theory of multicomplexes the same role played by the singular simplicial set in the context of simplicial sets.
If $X$ is a topological space,
the singular simplicial set  $\calS(X)$ of $X$ is the simplicial set having as simplices the singular simplices with values in $X$ (see e.g.~\cite{milnor-geom, Piccinini}). The simplices of the singular simplicial set are not embedded in general (for example, a non-constant singular $1$-simplex $\sigma \colon [0,1]\to X$ with $\sigma(0)=\sigma(1)$
gives rise to a non-embedded $1$-simplex of $\calS(X)$). 
The singular multicomplex $\mathcal{K}(X)$ is the multicomplex having as simplices
the singular simplices in $X$ with distinct vertices, up to affine symmetries. Therefore, the singular multicomplex differs from the singular simplicial set
both because of the requirement that singular simplices be injective on vertices, and because the geometric simplices of the singular multicomplex
do not come with a preferred ordering of their vertices (or of their faces).

The geometric realization of $\mathcal{K}(X)$ is a CW complex whose $0$-cells are in bijection with the points of $X$. Moreover, there is a natural continuous map
$
S_X \colon \lvert \mathcal{K}(X) \rvert \rightarrow X
$, which, under very mild conditions on the  space $X$,  is a weak homotopy equivalence
(for example, it is a homotopy equivalence provided that $X$ is a CW complex~\cite[Corollary 2.1.3]{FriMo}).

To the singular multicomplex $\calK(X)$ there is associated an aspherical multicomplex $\calA(X)$ which is defined as follows:
the $0$-skeleton of $\calA(X)$ coincides with the $0$-skeleton of $\calK(X)$ (hence, with $X$); the $1$-skeleton $\calA^1(X)$ of $\calA(X)$ is obtained by choosing
 one representative in every homotopy class (relative to the endpoints) of
$1$-simplices of $\calK(X)$; finally, if $G\subseteq \calA^1(X)$ is a graph isomorphic to the $1$-skeleton of the standard simplex $\Delta^n$,
then $\calA(X)$ contains exactly one $n$-dimensional simplex with $1$-skeleton $G$ if and only if there is some (possibly, more than one) $n$-simplex in $\calK(X)$
with $1$-skeleton $G$. One can define a simplicial projection
$$
\pi\colon \calK(X)\to \calA(X)
$$
that restricts to the identity of $\calK(X)^0=\calA(X)^0$, and induces an isomorphism on fundamental groups
(see~\cite[Chapter 3]{FriMo} for more details).

\section{The group $\Pi(X,X)$ and its action on $\calA(X)$}\label{am:sub:sec}
Let $X$ be a topological space. We recall the definition of 
the group $\Pi(X,X)$, which was first introduced by Gromov in~\cite{Grom82}.

\begin{defn}
Let $X_0$ be a subset of $X$. We define the set $\Omega(X,X_0)$ as follows. An element of $\Omega(X,X_0)$ is a family of paths $\{\gamma_x\}_{x\in X_0}$ satisfying the following conditions:
\begin{enumerate}
 \item each $\gamma_x\colon [0,1]\to X$
is a continuous path such that $\gamma_x(0)=x$ and $\gamma_x(1)\in X_0$;
\item  for all but finitely many $x\in X_0$, the path $\gamma_x$ is constant (i.e.~$\gamma_x(t)=x$ for every $t\in [0,1]$);
\item the map 
$$
X_0\to X_0\ ,\qquad x\mapsto \gamma_x(1)
$$
is a bijection of $X_0$ onto itself (which, by (2), is a permutation of $X_0$ with finite support).
\end{enumerate}

The usual concatenation of paths defines a multiplication on $\Omega(X,X_0)$, which endows $\Omega(X,X_0)$ with the structure of a  semigroup. 
In order
to obtain a group we consider the set $\Pi(X,X_0)$ of homotopy classes of elements of $\Omega(X,X_0)$, where two elements $\{\gamma_x\}_{x\in X_0}$, $\{\gamma'_x\}_{x\in X_0}$ of $\Omega(X,X_0)$ are said to be homotopic if 
$\gamma_x$ is homotopic to $\gamma_x'$ relative to the endpoints for every $x\in X_0$ (in particular, $\gamma_x(1)=\gamma_x'(1)$ for every $x\in\ X_0$).
\end{defn}

For any pair of subsets $U,V$ of $X$ with $U\subseteq V\subseteq X$ we have an obvious group homomorphism
$$
\Pi(U,V)\to \Pi(X,X)\ .
$$
We denote by $\Pi_X(U,V)$ the image of this homomorphism in $\Pi(X,X)$. 

\begin{lemma}[{\cite[Lemma 6.2.3]{FriMo}}]\label{prop-psi-pi-u-v-amenable}
Let $U$ be an amenable subset of $X$, and take any subset $V\subseteq U$. Then, the subgroup $\Pi_X(U,V)<\Pi(X,X)$ is amenable.
\end{lemma}

\subsection*{The action of $\Pi(X,X)$ on $\calA(X)$}
The group $\Pi(X,X)$ acts simplicially on $\calA(X)$ as follows.
Take $g\in\Pi(X,X)$, and let $\{\gamma_x\}_{x\in X}$ be a representative of $g$.
The set of vertices of $\calA(X)$ is canonically identified with $X$, and every element of $\Pi(X,X)$ induces a permutation (with finite support) of $X$.
Therefore, we can define the action of $g$ on the $0$-skeleton of $\calA(X)$ as the permutation induced by $g$ on $X$.

Let now $e$ be a $1$-simplex of $\calA(X)$ with endpoints $v_0,v_1\in \calA^0(X)=X$. By definition, $e$ represents a homotopy class
(relative to the endpoints) of paths in $X$ joining $v_0$ with $v_1$. Let $\gamma_e\colon [0,1]\to X$ be a representative of $e$ with $\gamma_e(0)=v_0$, $\gamma_e(1)=v_1$, and consider the concatenation 
$\gamma'\colon [0,1]\to X$ given by
$$
\gamma'=\gamma_{v_0}^{-1} * 
\gamma_e * \gamma_{v_1}\ ,
$$
where as usual we denote by $\gamma^{-1}$ the path $\gamma^{-1}(t)=\gamma(1-t)$. We then set $g\cdot e$ to be equal to the homotopy class (relative to the endpoints)
of the path $\gamma'$. It is easy to check that this definition of $g$ on the $1$-skeleton of $\calA(X)$ indeed extends the action of $g$ on the $0$-skeleton.
Moreover, it is proved in~\cite[Section 5.2]{FriMo} that the simplicial automorphism of $\calA(X)^1$ just described uniquely extends to a simplicial automorphism
of the whole $\calA(X)$, which will be denoted by $\psi(g)$. We have thus defined an action
$$
\psi\colon \Pi(X,X)\to \aut(\calA(X))\ .
$$

\begin{prop}[{\cite[Theorem 5.2.1]{FriMo}}]
For every $g\in \Pi(X,X)$, the automorphism $\psi(g)$ is simplicially homotopic to the identity. 
\end{prop}

 \section{Diffusion of chains}
Diffusion of chains was first introduced by Gromov in~\cite{Grom82}. 
In presence of suitable actions by amenable groups, diffusion allows to
decrease the $\ell^1$-norm of cycles without altering their homology class.

Let $\G$ be a group acting on a set $\Lambda$. 
For every function $f\colon \Lambda\to \R$ we denote by $\supp(f)$ the support of $f$, i.e.~the set
$$
\supp(f)=\{x\in \Lambda\, |\, f(x)\neq 0\}\ ,
$$
and
we denote by $\ell_0(\Lambda)$ the set of real functions on $\Lambda$ with finite support. For every $f\in \ell_0(\Lambda)$ we set
$$
\|f\|_1=\sum_{x\in \Lambda} |f(x)|\ .
$$

We also denote by $\calM_0(\G)$ the space of probability measures on $\G$ with finite support.
We will
often consider an element $\mu\in\calM_0(\G)$ as a non-negative function $\mu\in \ell_0(\G)$ such that $\sum_{g\in\G} \mu(g)=1$.

\begin{defn}
The \emph{diffusion operator} associated to a measure $\mu\in\calM_0(\G)$ is the $\R$-linear map
\begin{displaymath}
\mu \ast \colon \ell_{0}(\Lambda) \rightarrow \ell_{0}(\Lambda)\ ,\qquad
 (\mu \ast f) (x) = \sum_{g \in \, \G} \mu(g) f(g^{-1} x)\ .
\end{displaymath}
\end{defn}
It is readily seen that $\|\mu * f\|_1\leq \|f\|_1$ for every $f\in \ell_0(\Lambda)$, $\mu\in \calM_0(\G)$. Moreover,
the composition of diffusion operators is itself a diffusion operator, i.e.~for every $\mu_1,\mu_2\in \calM_0(\G)$ there exists 
$\mu\in \calM_0(\G)$ such that
$$
\mu_1*(\mu_2* f)=\mu * f
$$
(simply set $\mu(g)=\sum_{h\in \G} \mu_1(h)\mu_2(h^{-1}g)$). 

The following result provides the fundamental step in the proof that, under suitable conditions, diffusion decreases the $\ell^1$-norm of chains.
 
\begin{lemma}[{\cite[Corollary 8.1.6]{FriMo}}]\label{oss-scelta-epsilon}
Let $\G$ be an amenable group acting transitively on a set $\Lambda$, and take an element
  $f\in \ell_0(\Lambda)$.
Then for every $\eta>0$ there exists a probability measure $\mu\in\calM_0(\G)$ such that 
\begin{displaymath}
\lVert \mu \ast f \rVert_1 \leq  \Big| \sum_{x \in \Lambda} f(x) \Big| + \eta\ .
\end{displaymath}
\end{lemma}

We will apply the diffusion operator in the case when $\G$ is a subgroup of $\aut(K)$ 
for some multicomplex $K$, 
and $\Lambda$ is the set of algebraic $n$-simplices of $K$. 
More precisely,
let $\Theta(n)$ denote  the set of all $n$-algebraic simplices of $K$, and recall  
 that the simplicial chain module $C_n(K)$ is the free real vector space with basis $\Theta(n)$.
 There exists a natural isometric identification between the space $\ell_0(\Theta(n))$ and the chain module $C_n(K)$ (both endowed with their $\ell^1$-norms),
which
 identifies an element $f\in \ell_0(\Theta(n))$ with the simplicial chain $\sum_{\sigma \in \Theta(n)} f(\sigma)\cdot \sigma$. Therefore, if $\mu\in\calM_0(\G)$, 
 then the operator $\mu*\colon \ell_0(\Theta(n))\to\ell_0(\Theta(n))$ defines a diffusion operator on chains
 $$
 \mu*\colon C_n(K)\to C_n(K)\ .
 $$
An easy computation shows that, if $c = \sum_{\sigma\in\Theta(n)} a_\sigma\cdot \sigma$, then
$$\mu \ast c = 
\sum_{\gamma\in\G} \mu(\gamma)(\gamma\cdot c)\ .$$
 
 \begin{lemma}\label{homotopy:lemma}
 Suppose that every element of $\G$ is simplicially homotopic to the identity of $K$.
 There exists a constant $K_n$ only depending on $n$ such that the following condition holds.
 Let $c\in C_n(K)$ be a simplicial cycle, and take a probability measure $\mu\in\calM_0(\G)$. Then there exists
 $b\in C_{n+1}(K)$ such that 
 $$
 (\mu*c)-c=\partial b\ , \qquad \|b\|_1\leq K_n \|c\|_1\ .
 $$
  \end{lemma}
 \begin{proof}
 Since elements of $\G$ are simplicially homotopic to the identity, there exists a constant $K_n$ only depending on $n$ such that, for every $\gamma\in\G$, 
 $$\gamma\cdot c -c =b_\gamma$$
for some $b_\gamma\in C_{n+1}(K)$ such that $\|b_\gamma\|_1\leq K_n\|c\|_1$ (see e.g.~\cite[Remark 3.3.5]{FriMo}). Therefore, if $S=\supp(\mu)$ and $$b=\sum_{\gamma \in S} \mu(\gamma)b_\gamma\in C_{n+1}(K)\ ,$$ then
\begin{align*}
(\mu *c)-c=\left(\sum_{\gamma \in S} \mu(\gamma)\, \gamma\cdot c\right)- c=\sum_{\gamma \in S} \mu(\gamma)\left( \gamma\cdot c- c\right)=
\sum_{\gamma \in S}\mu(\gamma) b_\gamma=b\ ,
\end{align*} and
$$
\|b\|_1=\bigg\| \sum_{\gamma\in S} \mu(\gamma)b_\gamma\bigg\|_1\leq \sum_{\gamma\in S} \mu(\gamma)\|b_\gamma\|_1\leq
\sum_{\gamma\in S} \mu(\gamma)K_n=K_n\ .$$
 \end{proof}

\section{Proof of Theorem~\ref{main:thm}}
We are now ready to go into the proof of our main result. 
We first prove that it is not restrictive to restrict our attention to triangulable spaces.

\begin{lemma}\label{reduction:lemma}
 Suppose that Theorem~\ref{main:thm} holds whenever $X$ is (the topological realization of) a simplicial complex. Then Theorem~\ref{main:thm} holds for every topological space $X$.
\end{lemma}
\begin{proof}
Let $X$ be any topological space. The singular simplicial set $\calS(X)$ associated to $X$ 
is weakly homotopy equivalent to $X$ via a natural projection $j\colon |\calS(X)|\to X$ (see e.g.~\cite[Theorem 4.5.30]{Piccinini}). By subdividing $\calS(X)$ twice, one gets a simplicial complex
$\calS''(X)$ whose geometric realization $|\calS''(X)|$ is canonically homeomorphic to $|\calS(X)|$. Thus we have  a weak homotopy equivalence 
$$
j\colon |\calS''(X)|\to X\ ,
$$
which we still denote by $j$ with a slight abuse.

It is now sufficient to prove the following claims: 
\begin{enumerate}
 \item 
if $X$ admits an amenable cover of multiplicity $k$, then also $|\calS''(X)|$ admits an amenable cover
of multiplicity $k$;  
\item
if the map $\iota_n^\calS\colon H_n(|\calS''(X)|)\to H_n^{\ell^1}(|\calS''(X)|)$ is null,
then also the map 
$\iota_n^X\colon H_n(X)\to H_n^{\ell^1}(X)$ is null.
\end{enumerate}

In order to prove (1), observe that, if $\calU$ is an open cover of $X$, then $j^{-1}\calU=\{j^{-1}(U),\, U\in\calU\}$ is an open cover of $|\calS''(X)|$ such that $\mult(j^{-1}\calU)=\mult(\calU)$.
 Moreover,
 the map $j\colon |\calS''(X)|\to X$ induces an isomorphism on fundamental groups, so 
 from the commutative diagram
 $$
 \xymatrix{
 j^{-1}(U)\ar@{^{(}->}[r]\ar[d]_{j|_{j^{-1}(U)}}& |\calS''(X)|\ar[d]^{j}\\
 U\ar@{^{(}->}[r] & X
 }
 $$
 we deduce that the image of $\pi_1(j^{-1}(U),x_0)$ in $\pi_1(|\calS''(X)|,x_0)$ is isomorphic to a subgroup of the image of $\pi_1(U,j(x_0))$ in $\pi_1(X,j(x_0))$. Since every subgroup of an amenable group
 is amenable, this shows that 
 the cover $j^{-1}\calU$ is amenable if $\calU$ is so. This proves (1). 

In order to prove (2), let us consider the following commutative 
diagram:
 $$
\xymatrix{H_n(|\calS''(X)|)\ar[d]_{H_n(j)} \ar[r]^-{\iota_n^\calS} & H_n^{\ell^1}(|\calS''(X)|)\ar[d]^{H_n^{\ell^1}(j)}\\ H_n(X)\ar[r]^-{\iota^X_n} & H_n^{\ell^1}(X) \ .}
 $$
Since $j$ is a weak homotopy equivalence, the map $H_n(j)$ is an isomorphism (see e.g.~\cite[Proposition~4.21]{hatcher}). 
Thus, if $\iota_n^\calS$ is null, then also $\iota_n^X$ is null, whence the conclusion.
\end{proof} 

Thanks to the previous lemma, henceforth we assume that $X$
 is a topological space admitting a triangulation $T$, i.e.~that $X$ is equal to the geometric realization $|T|$ of a simplicial complex $T$.
Let $\calU=\{U_i\}_{i\in I}$ be an open amenable cover of $X$, and let $k=\mult(U)$.

For every vertex $v$ of $T$, the closed star of $v$ in $T$ is defined as the subcomplex of $T$ containing all the simplices containing $v$ (and all their faces). 
By suitably subdividing $T$ we may suppose that the following condition holds (see for instance \cite[Theorem~16.4]{munkres}): for every vertex $v$ of $T$ there exists $i(v)\in I$ such that the closed star of $v$ in $T$ 
is contained in the 
element $U_{i(v)}$ of the cover (of course, the choice of $i(v)$ may be non-unique).

For every $i\in I$ we set 
$$
V_i=\{v\in V(T)\, |\, i(v)=i\}\ .
$$
By construction we have $V_i\subseteq U_i$. Let us now set 
$$
\G=\bigoplus_{i\in I} \Pi_X(U_i,V_i)\ .
$$
The direct sum of amenable groups is amenable, so Lemma~\ref{prop-psi-pi-u-v-amenable} implies that $\G$ is amenable. Also observe that, if $i\neq j$, then $V_i\neq V_j$, so elements
in $\Pi_X(U_i,V_i)$ commute with elements in $\Pi_X(U_j,V_j)$. As a consequence, the group $\G$ naturally sits in $\Pi(X,X)$ as a subgroup, and acts on $\calA(X)$.

Observe that there is a canonical copy $\calA_T(X)$ of $T$ inside the multicomplex $\calA(X)$. 
In fact, the multicomplex $\calK(X)$ contains a submulticomplex 
 $\calK_T(X)\cong T$ whose simplices
are the equivalence classes of the affine parametrizations of simplices of $T$. Since $T$ is a simplicial complex, every simplex in $\calK_T(X)$
is uniquely determined by its $0$-skeleton. 
As a consequence, the projection $\pi\colon \calK(X)\to\calA(X)$
restricts to an injective map on $\calK_T(X)$. We may thus set $\calA_T(X)=\pi(\calK_T(X))$.

Let us now fix $n\geq k=\mult(\calU)$, and let $\alpha$ be an element in $H_n(X)$. 
Recall that there exists a canonical isomorphism  $H_n(\phi_*)\colon H_n(T)\to H_n(|T|)=H_n(X)$
between the simplicial homology of $T$  the singular
homology of its geometric realization (see Proposition~\ref{simpl:iso}). We denote by $\alpha_T\in H_n(T)$ the class such that $H_n(\phi_*)(\alpha_T)=\alpha$, 
by $\alpha_\calK\in H_n(\calK(X))$ the image of $\alpha_T$ under the map induced by the inclusion
$T\cong \calK_T(X)\hookrightarrow \calK(X)$, and we set
$\alpha_\calA=H_n(\pi)(\alpha_\calK)\in H_n(\calA(X))$. It readily follows from the definitions that
\begin{equation}\label{aaaeee}
\alpha=H_n(S_X)(H_n(\varphi_*)(\alpha_\calK))\ ,
\end{equation}
where $S_X\colon |\calK(X)|\to X$ is the natural projection.

We denote by 
$$
\iota_n^X\colon H_n(X)\to H_n^{\ell^1}(X)\, ,\qquad \iota_n^\calA\colon H_n(\calA(X))\to H_n^{\ell^1}(\calA(X))
$$
the maps induced by the inclusion of ordinary (singular or simplicial) chains into $\ell^1$-chains.

\begin{prop}\label{bastaA}
 Suppose that $\iota_n^\calA(\alpha_\calA)=0$ in $H_n^{\ell^1}(\calA(X))$. Then $\iota_n^X(\alpha)=0$ in $H_n^{\ell^1}(X)$.
\end{prop}
\begin{proof}
Let us consider the following commutative diagram:
$$
\xymatrix{
H_n(\calK(X))\ar[rr]^{H_n(\pi)}\ar[d]^{\iota_n^\calK} & & H_n(\calA(X))\ar[d]^{\iota_n^\calA}\\
H_n^{\ell^1}(\calK(X))\ar[rr]^{H_n^{\ell^1}(\pi)}& & H_n^{\ell^1}(\calA(X))\ .
} 
$$
Since $H_n(\pi)(\alpha_\calK)=\alpha_\calA$ and 
$\iota_n^\calA(\alpha_\calA)=0$, we have 
$H_n^{\ell^1}(\pi)(\iota_n^\calK (\alpha_\calK))=0$.
By~\cite[Corollary 4.4.4]{FriMo}, the map $H^n_b(\pi)\colon H^n_b(\calA(X))\to H^n_b(\calK(X))$ is an isometric isomorphism,
hence  L\"oh's Translation Principle~\cite[Theorem 1.1]{Loeh} implies that also the map $H_n^{\ell^1}(\pi)$ is an isomorphism.
We thus have $\iota_n^\calK(\alpha_\calK)=0$.

Let us now consider the following commutative diagram:
$$
\xymatrix{
H_n(\calK(X))\ar[d]^{\iota_n^\calK}\ar[rr]^{H_n(\phi_*)} & & H_n(|\calK(X)|)\ar[d]\ar[rr]^{H_n(S_X)} & & H_n(X)\ar[d]^{\iota_n^X}\\
H_n^{\ell^1}(\calK(X))\ar[rr]^{H_n(\phi_*^{\ell^1})} & & H_n^{\ell^1}(|\calK(X)|)\ar[rr]^{H_n^{\ell^1}(S_X)} & & H^{\ell^1}_n(X)\ .
}
$$
Putting together equation~\eqref{aaaeee} and the fact that $\iota_n^\calK(\alpha_\calK)=0$ we have
$$
\iota^X_n(\alpha)=\iota_n^X(H_n(S_X)(H_n(\varphi_*)(\alpha_\calK)))=H_n^{\ell^1}(S_X)(H_n(\varphi_*^{\ell^1})(\iota^\calK_n(\alpha_\calK)))=0\ .
$$
\end{proof}

\subsection*{The submulticomplex $\calA'_T(X)$} 
We are  interested in diffusing chains supported on the submulticomplex $\calA_T(X)\subseteq \calA(X)$. 
However, $\calA_T(X)$ is \emph{not} left invariant by the action of $\G$ on $\calA(X)$, hence we define the submulticomplex
$$
\calA_T'(X)=\Gamma\cdot \calA_T(X)\ \subseteq \ \calA(X)\ .
$$
By construction, $\calA_T'(X)$ contains $\calA_T(X)$ and is $\Gamma$-invariant.
Recall that the $0$-skeleton of $\calA(X)$ coincides with the $0$-skeleton of $\calK(X)$, which is canonically identified with
$X$ itself. In particular, it makes sense to say that a vertex of a simplex of $\calA_T(X)$ (and of $\calA'_T(X)$) belongs to $V_i$ for some $i\in I$. 

\begin{lemma}\label{orientation-lemma}
 Let $\Delta$ be an $n$-dimensional simplex of $\calA_T'(X)$, where $n\geq k=\mult(\calU)$. Then there exists
 $g\in\Gamma$ such that $g(\Delta)=\Delta$, and the restriction of $g$ to $\Delta$ is orientation-reversing
 (i.e.~it induces an odd permutation of the vertices of $\Delta$).
\end{lemma}
\begin{proof}
 Observe first that we may assume $\Delta\subseteq \calA_T(X)$. In fact, by definition there exists
 $h\in \G$ such that $\Delta=h(\Delta')$, where $\Delta'\subseteq \calA_T(X)$. If $g\in\G$ leaves $\Delta'$ invariant
 and is orientation reversing on $\Delta'$, then the automorphism $hgh^{-1}\in \G$ leaves $\Delta$ invariant
 and is orientation reversing on $\Delta$.
 
Therefore, let $v_0,\ldots,v_n$ be the vertices of $\Delta$, where $\Delta$ is an $n$-simplex of $T\cong \calA_T(X)$. 
By construction, the geometric simplex $|\Delta|\subseteq |T|=X$ satisfies
$$
|\Delta|\subseteq \bigcap_{j=0}^n U_{i(v_j)}\ .
$$
In particular, the intersection $ \bigcap_{j=0}^n U_{i(v_j)}$ is non-empty. Since $\mult(\calU)< n+1$,
this implies that 
there exists $i_0\in I$ such that at least
 two vertices of $\Delta$ belong to the same $V_{i_0}$. Up to reordering, we may assume that $v_0\in V_{i_0}$, $v_1\in V_{i_0}$. For every $h,k\in \{0,\ldots,n\}$,
$h\neq k$, let $e_{hk}\colon [0,1]\to X$ be the affine parametrization of the edge of $|\Delta|\subseteq |T|=X$ starting at $v_h$ and ending at $v_k$.  Since the closed stars
of $v_0$ and $v_1$ are contained in $U_{i_0}$, the paths $\{e_{01},e_{10}\}$ define an element of $g\in \Pi_X(U_{i_0},V_{i_0})<\G$. For every $h=2,\ldots,n$,
the concatenation $e_{01}*e_{1h}$ is homotopic relative to the endpoints to $e_{0h}$, and the concatenation
$e_{10}*e_{0h}$ is homotopic relative to the endpoints to $e_{0h}$. Thus, 
$g$ leaves the $1$-skeleton of $\Delta$ invariant, and acts on the vertices of $\Delta$ by switching $v_0$ with $v_1$, while
leaving all the other vertices fixed. Since simplices of $\calA(X)$ are uniquely determined by their $1$-skeleton,
this implies that $g(\Delta)=\Delta$. Moreover, $g$ acts on the vertices of $\Delta$ as a transposition, whence the conclusion.
 \end{proof}

  \subsection*{Alternating chains}
   Let $c\in C_n(\calA(X))$ be a simplicial chain, i.e.~let
  $c=\sum_{\sigma\in\Theta(n)} a_\sigma\cdot \sigma\in C_n(\calA(X))$, where $a_\sigma=0$ for all but a finite number of $\sigma\in \Theta(n)$. 
  We say that $c$ is \emph{alternating} if the following condition holds:
  if $\sigma=(\Delta,(v_0,\ldots,v_n))$, $\sigma'=(\Delta,(v_{\tau(0)},\ldots,v_{\tau(n)}))$ are algebraic simplices
which can be obtained one from the other via a permutation $\tau$ of their vertices, then $a_\sigma=\varepsilon(\tau)a_{\sigma'}$
(here $\varepsilon(\tau)=\pm 1$ denotes the sign of $\tau$).

The linear operator
$\alt_*\colon C_*(\calA(X))\to C_*(\calA(X))$ such that
$$
\alt(\Delta,(v_0,\ldots,v_j))=\frac{1}{(j+1)!} \sum_{\tau\in\mathfrak{S}_{j+1}} \varepsilon(\tau)(\Delta,(v_{\tau(0)},\ldots,v_{\tau(j)})) 
$$
is well defined and homotopic to the identity. 
Therefore, any homology class in $H_*(\calA(X))$ may be represented by an alternating cycle.

\subsection*{Diffusion of chains}
The fundamental step in the proof of our main theorem is the following:

\begin{thm}\label{Thm:pre:van:toy:ex}
Let $c\in\calA(X)$ be an alternating $n$-cycle supported on $\calA'_T(X)$. 
Then there exists a probability measure $\mu\in\calM_0(\G)$ such that 
$$
\|\mu* c\|_1 \ \leq\ \frac{\|c\|_1}{2}\ .
$$
\end{thm}
\begin{proof}
 Since $c$ is a finite linear combination of simplices in $\Theta(n)$, there exist $\G$-orbits 
$ \Theta(n)_1,\ldots, \Theta(n)_s$ for the action of $\G$ on $\Theta(n)$ such that 
$$c = \sum_{j = 1}^{s} c_j\ ,$$
where
$$c_j = \sum_{\sigma\in\Theta(n)_j} a_\sigma \sigma$$
for every $j=1,\ldots,s$.

Of course we may suppose $c\neq 0$, otherwise the statement is trivial. Let us set $\eta = \|c\|_1/(2 s)>0$. We now show by induction that there exist probability measures $\mu_1,\ldots,\mu_s\in\calM_0(\G)$ 
satisfying the following property:
for every $j=1,\ldots, s$ and every $1\leq i\leq j$
\begin{equation}\label{inductivemu}
\|\mu_j*(\mu_{j-1}*(\ldots*(\mu_1*c_i)))\|_1\leq \eta\ .
\end{equation}
Let $i=j=1$, and
let $f_1\in \ell_0(\Theta(n))$ be the function associated to $c_1$ (by construction, $f_1$ is supported on $\Theta(n)_1$).
Since $\supp(f_1)$ is finite and contained in a single $\G$-orbit, we can  apply Lemma~\ref{oss-scelta-epsilon} to  $f_1$ (now considered as a function on $\Theta(n)_1$), thus
obtaining a   measure $\mu_1\in\calM_0(\G)$ such that 
\begin{equation}\label{formula:diffusione:toy}
\| \mu_1\ast c_1\|_1=\| \mu_1 \ast f_1 \|_1 \leq \Big\lvert \sum_{\sigma \in \, \Theta(n)_1} f_1(\sigma) \Big\rvert + \eta\ .\end{equation}

Since $c$ is supported on $\calA'_T(X)$, 
Lemma~\ref{orientation-lemma} implies that 
 for any algebraic simplex 
$\sigma = (\Delta,(v_0,\ldots,v_n)) \in \, \Theta(n)_1$
there exists an element $g\in \Gamma$ 
such that $g\cdot \sigma = (\Delta,(v_{\tau(0)},\ldots,v_{\tau(n)}))$, where $\tau$ is an odd permutation. 
Using that $c$ is alternating,
this easily implies that 
$\sum_{\sigma \in \, \Theta(n)_1} f_1(\sigma)= 0$ which, together with~\eqref{formula:diffusione:toy}, gives 
\begin{equation}\label{formula:diffusione:toy2}
\| \mu_1 \ast c_1 \|_1 \leq \eta \ .\end{equation}
This settles the case $j=1$.

Suppose now we have constructed probability measures $\mu_1,\ldots,\mu_l$ satisfying property~\eqref{inductivemu}, and let
$h=\mu_l*(\mu_{l-1}*(\ldots*(\mu_1*f_{l+1})))$. By construction $\supp(h)$ is finite and contained in $\Theta(n)_{l+1}$. Moreover, since
diffusion preserves the alternation of chains and $\calA'_T(X)$ is $\G$-invariant, the finite chain associated to $h$ is alternating and supported on $\calA'_T(X)$. This allows us to argue as above to 
obtain a probability measure $\mu_{l+1}\in\calM_0(\G)$ such that $\|\mu_{l+1}\ast g\|_1\leq \eta$, which implies 
\begin{align*}
\|\mu_{l+1}*(\mu_{l}*(\ldots*(\mu_1*c_{l+1})\ldots))\|_1
& =\|\mu_{l+1}*(\mu_{l}*(\ldots*(\mu_1*f_{l+1})\ldots))\|_1\\ &=
\|\mu_{l+1}\ast g\|_1\leq \eta\ .
\end{align*}
Moreover, for every $1\leq i\leq l$, if $c'=\mu_i*(\mu_{i-1}*(\ldots*(\mu_1*c_i)))$ then we know by our inductive hypothesis
that $\|c'\|_1\leq \eta$, hence
$$
\|\mu_{l+1}*(\mu_{l}*(\ldots*(\mu_1*c_{i})))\|_1=\|\mu_{l+1}*(\mu_{l}*(\ldots*(\mu_{i+1}*c')\ldots))\|_1\leq \|c'\|_1\leq \eta\ ,
$$
where the second-last inequality is due to the fact that diffusion is always norm non-increasing. This proves the inductive step,
hence inequality~\eqref{inductivemu} for every $j=1,\ldots, s$ and every $1\leq i\leq j$.

Let now $\mu\in\calM_0(\G)$ be such that
$$\mu*c=\mu_s*(\mu_{s-1}*(\ldots*(\mu_1*c)\ldots ))\ .$$ 
By~\eqref{inductivemu} we have
\begin{align*}
\|\mu*c\|_1&=\Big\|\mu_s*\Big(\mu_{s-1}*\Big(\ldots*\Big(\mu_1*\Big(\sum_{j=1}^s c_j\Big)\Big)\ldots\Big)\Big)\Big\|_1\\ &=
\Big\|\sum_{j=1}^s\mu_s*\big(\mu_{s-1}*\big(\ldots*\big(\mu_1* c_j\big)\ldots\big)\big)\big\|_1 \\ &\leq
\sum_{j=1}^s \big\|\mu_s*\big(\mu_{s-1}*\big(\ldots*\big(\mu_1* c_j\big)\ldots\big)\big)\big\|_1\\ &\leq s\cdot \eta\leq \frac{\|c\|_1}{2}\ .
\end{align*}
\end{proof}

We are now ready to conclude the proof of Theorem~\ref{main:thm}. 
Recall that, starting from a generic element $\alpha\in H_n(X)$, we constructed an element $\alpha_\calA\in H_n(\calA(X))$ and proved that
$\iota_n^X(\alpha)=0$ in $H_n^{\ell^1}(X)$ provided that $\iota_n^\calA(\alpha_\calA)=0$ in $H_n^{\ell^1}(\calA(X))$ (see Proposition~\ref{bastaA}).
More precisely, the class $\alpha_\calA$ is the image of a class $\alpha_T\in H_n(T)$ under the map induced on simplicial homology by the inclusion
$T\cong \calA_T(X)\hookrightarrow \calA(X)$.  In particular, we may assume that $\alpha_\calA$ is represented by a cycle $c\in C_n(\calA(X))$ supported
on $\calA_T(X)$ (hence, on $\calA'_T(X)$). By applying the operator $\alt_n$ described above, we can also suppose that $c$ is alternating (of course, 
if $c$ is supported on $\calA'_T(X)$ then also $\alt_n(c)$ is supported on $\calA'_T(X)$).

We now 
inductively define
cycles $c_i\in C_n(\calA'_T(X))$, and chains $b_i\in C_{n+1}(\calA(X))$, $i\in\mathbb{N}$, such that the following conditions hold:
\begin{enumerate}
\item $c_0=c$;
\item $c_i$ is alternating;
\item $\|c_i\|_1\leq 2^{-i} \|c\|_1$;
\item $c_i-c_{i+1}=\partial b_i$;
\item $\|b_i\|_1\leq K_n 2^{-i} \|c\|_1$, where $K_n$ is a universal constant.
\end{enumerate}
For $i=0$ we just set $c_0=c$, and there is nothing to prove. Suppose we have constructed $c_0,\ldots,c_i$ satisfying the properties listed above. 
By Theorem~\ref{Thm:pre:van:toy:ex}, there exists a measure $\mu\in \calM_0(\G)$ such that $$\|\mu*c_i\|_1\leq \frac{\|c_i\|}{2}\leq 2^{-i-1} \|c\|_1 \ .$$
We then set $c_{i+1}=\mu*c_i$, and (3) is satisfied.
The diffusion of an alternating chain is alternating, hence also (2) holds. Finally, Lemma~\ref{homotopy:lemma} provides a chain $b_i\in C_{n+1}(\calA(X))$ satisfying
(4) and (5).

Let us now consider the infinite chain $$\overline{b}=\sum_{i\in\mathbb{N}} b_i\ .$$
By (5), the $\ell^1$-norm of $\overline{b}$ is finite, hence $\overline{b}\in C_{n+1}^{\ell^1}(\cal A(X))$. Moreover,
using (4) it is easy to check that
$$
c=c_0=\partial \overline{b}\ .
$$
This shows that $\iota_n(\alpha_\calA)=\iota_n([c])=0$ in $H_n^{\ell^1}(\calA(X))$.
Thanks to Proposition~\ref{bastaA}, this concludes the proof
of Theorem~\ref{main:thm}.

\begin{rem}\label{aspherical:rem}
In this paper, bounded cohomology was exploited (via L\"oh's Translation Principle) 
 to reduce the vanishing of $\iota_n^X(\alpha)$ to the vanishing of $\iota_n^\calA(\alpha_\calA)$ (see Proposition~\ref{bastaA}). After that, our constructions only involved diffusion of cycles. 
 In fact, if $X$ is an aspherical CW complex, then $X$, $|\calK(X)|$ and $|\calA(X)|$ are homotopy equivalent. In this case, our main theorem can be obtained via diffusion of cycles only, without referring to bounded cohomology.
\end{rem}

\bibliographystyle{amsalpha}
\bibliography{biblionote}

\end{document}